\documentclass[12pt, reqno]{amsart}
\makeatletter
\@namedef{subjclassname@1991}{$\mathrm{1991}$ Mathematics Subject Classification}
\@namedef{subjclassname@2000}{$\mathrm{2000}$ Mathematics Subject Classification}
\@namedef{subjclassname@2010}{$\mathrm{2010}$ Mathematics Subject Classification}
\@namedef{subjclassname@2020}{$\mathrm{2020}$ Mathematics Subject Classification}
\makeatother
\usepackage{amsmath,amsthm, amscd, amsfonts, amssymb, graphicx, color}
\usepackage[bookmarksnumbered, colorlinks, plainpages,linkcolor=blue,urlcolor=blue,citecolor=blue]{hyperref}
\newtheorem{thm}{Theorem}[section]
\newtheorem{cor}[thm]{Corollary}
\newtheorem{lem}[thm]{Lemma}

\newtheorem{defn}[thm]{Definition}

\numberwithin{equation}{section}

\begin{document}

\title{$m$-weak group inverse in a ring with proper involution}

\author{Huanyin Chen}
\address{School of Big Data, Fuzhou University of International Studies and Trade, Fuzhou 350202, China}
\email{<huanyinchenfz@163.com>}


\begin{abstract}The $m$-weak group inverse was recently studied in the literature. The purpose of this paper is to investigate new properties of this generalized inverse for ring elements. We introduce the $m$-weak group decomposition for a ring element and prove that it coincides with its $m$-weak group invertibility. We present the equivalent characterization of the $m$-weak group inverse by using a polar-like property.
The relations between $m$-weak group inverse and core-EP inverse are also established. These give some new properties of the weak group inverse for complex matrices and ring elements.
\end{abstract}

\maketitle

\section{Introduction}

Let $R$ be an associative ring with an identity. An element $a$ in $R$ has Drazin inverse if there exists $x\in R$ such that $$ax^2=x, ax=xa, a^k=xa^{k+1}$$ for some $k\in {\Bbb N}$. Such $x$ is unique, if exists, and denote it by $a^D$. The least $k$ satisfying the preceding equations is called the Drazin index of $a$ and denoted by $ind(a)$. An element $a\in R$ has group inverse provided that $a$ has Drazin inverse with $ind(a)=1$, that is, there exists $x\in R$ such that $$ax^2=x, ax=xa, a=a^2x.$$ Such $x$ is unique if exists, denoted by $a^{\#}$, and called the group inverse of $a$.
Evidently, a square complex matrix $A$ has group inverse if and only if $rank(A)=rank(A^2)$.

A ring is called a *-ring if there exists an involution $*: x\to x^*$ satisfying $(x+y)^*=x^*+y^*, (xy)^*=y^*x^*$ and $(x^*)^*=x$.
The involution $*$ is proper if $x^*x=0\Longrightarrow x=0$ for any $x\in R$. Let ${\Bbb C}^{n\times n}$ be the Banach algebra of all $n\times n$ complex matrices, with conjugate transpose $*$ as the involution. Then the involution $*$ is proper.

Let $R$ be a ring with a proper involution $*$. An element $a\in R$ has weak group inverse if there exist $x\in R$ and $n\in \Bbb{N}$ such that $$ax^2=x, (a^*a^2x)^*=a^*a^2x, a^k=xa^{k+1}$$ for some $k\in {\Bbb N}$. If such $x$ exists, it is unique, and denote it by $a^{\overline{W}}$.

An element $a\in R$ has core-EP inverse if there exist $x\in R$ and $k\in \Bbb{N}$ such that $$ax^2=x, (ax)^*=ax, a^k=xa^{k+1}.$$ If such $x$ exists, it is
unique, and denote it by $a^{\overline{D}}$. Evidently, a square complex matrix $A$ has weak group inverse $X$ if it satisfies the system of equations: $$AX^2=X, AX=A^{\overline{D}}A.$$ For more details of the group inverse and weak group inverse and core-EP inverse, readers can see
(see~[D2,GC,M,MD1,MD2,MD5,MD4,W,W1,W2,Y,Z1,Z3]).

In [23], Zhou et al. introduced and studied the $m$-weak group inverse in a ring with proper involution.
.
\begin{defn} An element $a\in \mathcal{A}$ has $m$-generalized group inverse if there exist $x\in \mathcal{A}$ and $k\in {\Bbb N}$ such that $$ax^2=x, xa^{k+1}=a^k~\mbox{and} ~(a^k)^*a^{m+1}x=(a^k)^*a^m.$$ The preceding $x$ is called the $m$-generalized group inverse of $a$, and denoted by $a^{\overline{W}_m}$.\end{defn}

Evidently, a square complex matrix $A$ has $m$-weak group inverse $X$ if it satisfies the system of equations: $$AX^2=X, AX=(A^{\overline{\dag}})^mA^m $$ (see~[7]).

The motivation of this paper is to investigate new properties of $m$-weak group inverse for a ring element. In Section 2, we introduce the $m$-weak group decomposition for a ring element.

\begin{defn} An element $a\in R$ has $m$-weak group decomposition if there exist $x,y\in R$ such that $$a=x+y, x^*a^{m-1}y=yx=0, x\in
R^{\#}, y\in R^{nil}.$$\end{defn} Here $R^{nil}$ denote the set of nilpotents in $R$. We prove that the $m$-weak group decomposition of a ring element coincides with its $m$-weak group invertibility. Additive property and reverse order law of the $m$-weak group inverse are presented.

In Section 3, we investigate some equivalent characterizations of $m$-weak group inverse. We characterize the $m$-weak group inverse by using a polar-like property.

In Section 4 we establish relations between $m$-weak group inverse and core-EP inverse. In addition, we discuss a canonical form of the $m$-weak group inverse of a ring element by means of its core-EP inverse.

Throughout the paper, all rings are associative with a proper involution $*$. We use $R^{D}, R^{\overline{D}}, R^{\overline{W}}$ and
$R^{\overline{W}_m}$ to denote the sets of all Drazin invertible, core-EP invertible, weak group invertible and $m$-weak group invertible elements in the ring $R$, respectively. $\mathcal{R}(X)$ represents the range space of a complex matrix $X$.

\section{algebraic properties}

The purpose of this section is to investigate various algebraic properties of the $m$-weak group inverse. We characterize this generalized inverse by new element-wise properties.

\begin{thm} Let $a\in R$. Then the following are equivalent:\end{thm}
\begin{enumerate}
\item [(1)] $a\in R$ has $m$-weak group inverse.
\item [(2)] $a$ has the $m$-weak group decomposition, i.e., there exist $x,y\in R$ such that $$a=x+y, x^*a^{m-1}y=yx=0, x\in
R^{\#}, y\in R^{nil}.$$
\item [(3)] $a\in R^D$ and there exists $x\in R$ such that $$x=ax^2, (a^D)^*a^{m+1}x=(a^D)^*a^m, a^n=xa^{n+1}~\mbox{for some} ~n\in {\Bbb N}.$$
\end{enumerate}
\begin{proof} $(1)\Rightarrow (2)$ By hypotheses, we have $z\in R$ such that $$z=az^2, ((a^m)^*a^{m+1}z)^*=(a^m)^*a^{m+1}z, a^n=za^{n+1}~\mbox{for some}~n\in {\Bbb N}.$$

Set $x=a^2z$ and $y=a-a^2z.$ Then $a=x+y$.
We check that $$\begin{array}{rcl}
(a-za^2)z&=&(a-za^2)az^2\\
&=&(a-za^2)a^2z^3\\
&\vdots&\\
&=&(a-za^2)a^{n-1}z^n\\
&=&(a^n-za^{n+1})z^n=0.
\end{array}$$
Therefore $(a-za^2)z=0$, i.e., $az=za^2z$. This implies that $z=(az)z=(za^2z)z=z(a^2z^2)=zaz$.

Claim 1. $x$ has group inverse. Evidently, we verify that
$$\begin{array}{rll}
zx^2&=&za^2za^2z=az(a^2z)=a(za^2z)=a^2z=x,\\
xz^2&=&(a^2z)z^2=(a^2z^2)z=az^2=z,\\
xz&=&a^2z^2=az=za^2z=zx.
\end{array}$$ Hence, $x\in R^{\#}$ and $z=x^{\#}$.

Claim 2. $y\in R^{nil}$. Obviously, $(a-a^2z)a^{n+1}=a^{n+2}-a^2(za^{n+1})=a^{n+2}-a^2a^n=0$.
Then we have $(a-a^2z)a^2z=(a-a^2z)a^{n+1}z^n=0$. Accordingly, $(a-a^2z)(a-a^2z)^{n+1}=(a-a^2z)a^{n+1}=0,$ and therefore
$y^{n+1}=0$. That is, $y\in R^{nil}$.

Claim 3. We verify that $$\begin{array}{rl}
&x^*a^{m-1}y\\
=&(a^2z)^*a^{m-1}(a-a^2z)=[a(az)]^*a^{m-1}(a-a^2z)\\
=&[a(a^mz^m)]^*a^{m-1}(a-a^2z)=(z^{m-1})^*[(a^{m+1}z)^*a^m](1-az)\\
=&(z^{m-1})^*[(a^{11})^*a^{m+1}z]^*(1-az)\\
=&(z^{m-1})^*[(a^{11})^*a^{m+1}z](1-az)\\
=&(z^{m-1})^*(a^m)^*a^{m+1}(z-zaz)=0,\\
&yx=(a-a^2z)(a^2z)\\
=&a^3z-a^2(za^2z)=a^3z-a^2(az)=0.
\end{array}$$ Therefore $a=x+y$ is a $m$-weak group decomposition.

$(2)\Rightarrow (3)$ By hypothesis, $a$ has the $m$-weak group decomposition $a=a_1+a_2$.
Here, $a_1\in R^{\#}, a_2\in R^{nil}, a_1^*a^{m-1}a_2=0$ and $a_2a_1=0$.
Let $x=(a_1)^{\#}$. Then $ax^2=(a_1+a_2)[(a_1)^{\#}]^2=(a_1)^{\#}=x$.

Set $n=ind(a_2)$. Since $a_2a_1=0$, we have $a-xa^2=(a_1+a_2)-(a_1^{\#}a_1+a_1^{\#}a_2)(a_1+a_2)=
(1-a_1^{\#}a_1-a_1^{\#}a_2)a_2$, and so
$$\begin{array}{rl}
&a^n-xa^{n+1}=(a-xa^2)a^{n-1}\\
=&(1-a_1^{\#}a_1-a_1^{\#}a_2)a_2(a_1+a_2)^{n-1}\\
=&[1-a_1^{\#}a_1-a_1^{\#}a_2]a_2^{n}=0.
\end{array}$$
Then $a^n=xa^{n+1}.$ Since $a_2a_1=0, a_1a_1^{\pi}=0$ and $a_2^D=0$, it follows by ~[9, Theorem 2.1] that $a\in R^D$ and $$a^D=a_1^{\#}+\sum\limits_{k=1}^{n-1}(a_1^{\#})^{k+1}a_2^k.$$
Then $$\begin{array}{rll}
(a^D)^*a^{m-1}a_2&=&(a_1^{\#})^*a^{m-1}a_2+\sum\limits_{k=1}^{n-1}[(a_1^{\#})^{k+1}a_2^k]^*a^{m-1}a_2\\
&=&[(a_1^{\#})^2]^*(a_1^*a^{m-1}a_2)+\sum\limits_{k=1}^{n-1}[(a_1^{\#})^{k+2}a_2^k]^*[a_1^*a^{m-1}a_2]\\
&=&0;
\end{array}$$ whence, $(a^D)^*a^{m-1}a_1=(a^D)^*a^m$.
Therefore $$\begin{array}{rll}
(a^D)^*a^{m+1}x&=&(a^D)^*a^m(a_1+a_2)a_1^{\#}\\
&=&(a^D)^*a^{m-1}(a_1^2a_1^{\#})=(a^D)^*a^{m-1}a_1=(a^D)^*a^m.
\end{array}$$

$(3)\Rightarrow (1)$ By hypothesis, $a\in R^D$ and there exists $x\in R$ such that $$x=ax^2, (a^D)^*a^{m+1}x=(a^D)^*a^m, a^n=xa^{n+1}~\mbox{for some} ~n\in {\Bbb N}.$$ Write $a^k=a^Da^{k+1}$ for some $k\in {\Bbb N}$. Then $(a^{k+1})^*(a^D)^*a^{m+1}x=(a^{k+1})^*(a^D)^*a^m$, and so
$(a^Da^{k+1})^*a^{m+1}x=(a^Da^{k+1})^*a^m$. Hence $(a^k)^*a^{m+1}x=(a^k)^*a^m$.
Therefore $a\in R$ has $m$-weak group inverse, as asserted.\end{proof}

\begin{cor} Let $a\in R^{\overline{W}_m}$. Then the following hold.\end{cor}
\begin{enumerate}
\item [(1)] $a^{\overline{W}_m}=a^{\overline{W}_m}aa^{\overline{W}_m}$.
\vspace{-.5mm}
\item [(2)] $aa^{\overline{W}_m}=a^{\overline{W}_m}a^2a^{\overline{W}_m}$.
\vspace{-.5mm}
\item [(3)] $aa^{\overline{W}_m}=a^m(a^{\overline{W}_m})^m$ for any $m\in {\Bbb N}$.
\end{enumerate}
\begin{proof} These are obvious by the proof of Theorem 2.1.\end{proof}

We come now to provide a new characterization for the $m$-weak group inverse of a complex matrix.

\begin{cor} Let $A,X\in {\Bbb C}^{n\times n}$. Then the following are equivalent:\end{cor}
\begin{enumerate}
\item [(1)] $A^{\overline{W}_m}=X$.
\item [(2)] $X=(A^{\overline{\dag}})^{m+1}A^m$.
\item [(3)] $X=(A^{\overline{W}})^mA^{m-1}$.
\item [(4)] There exist $Y,Z\in {\Bbb C}^{n\times n}$ such that $$A=Y+Z, Y^*A^{m-1}Z=ZY=0, Y~\mbox{has group inverse}, Z~\mbox{is nilpotent}.$$
\item [(5)] There exists $k\in {\Bbb N}$ such that
$$A^k=XA^{k+1}, X=AX^2, (A^D)^*A^{m+1}X=(A^D)^*A^m.$$
\end{enumerate}
\begin{proof} $(1)\Leftrightarrow (2)\Leftrightarrow (3)$ are proved in ~[7, Theorem 3.1 and Theorem 4.3].

$(1)\Leftrightarrow (4)$ are obtained by Theorem 2.1.

$(1)\Rightarrow (5)$ Clearly, $X=A^{\overline{g}_m}$. Then we have
$$X=AX^2, (A^D)^*A^{m+1}X=(A^D)^*A^m.$$ Moreover, there exists some $k\in {\Bbb N}$ such that $A^k=A^{\overline{W}}A^{k+1}=XA^{k+1}$, as desired.

$(5)\Rightarrow (1)$ By hypothesis, there exists $k\in {\Bbb N}$ such that
$$A^k=XA^{k+1}, X=AX^2, (A^D)^*A^{m+1}X=(A^D)^*A^m.$$ As in the proof in Corollary 3.4, we prove that
$A-XA^2$ is nil, and so $A\in {\Bbb C}^{n\times n}$ has $m$-generalized group inverse. Therefore it has $m$-weak group inverse, as asserted.
\end{proof}

We are ready to prove:

\begin{thm} Let $a\in R^{\overline{W}_m}$. Then $(a^{\overline{W}_m})^{\overline{W}_m}=a^2a^{\overline{W}_m}.$
\end{thm}
\begin{proof} Let $x=a^{\overline{W}_m}$.

Claim 1. $x(a^2x)^2=a^2x$.

We verify that $x(a^2x)^2=(xa^2x)a^2x=a(xa^2x)=a(ax)=a^2x$.

Claim 2. $[(x^m)^*x^{m+1}(a^2x)]^*=(x^m)^*x^{m+1}(a^2x)$.

It is easy to check that $$(x^m)^*x^{m+1}(a^2x)=(x^m)^*x^m(xa^2)x=(x^m)^*x^{m-1}(xax)=(x^m)^*x^m.$$ Hence, we derive
$$[(x^m)^*x^{m+1}(a^2x)]^*=[(x^m)^*x^m]^*=(x^m)^*x^m=(x^m)^*x^{m+1}(a^2x),$$ as required.

Claim 3. $(a^2x)x^{n+1}=x^n$.

We see $(a^2x)x^{n+1}=a^2x^{n+2}=a(ax^2)x^n=(ax)x^n=ax^{n+1}=(ax^2)x^{n-1}=x^n$, as claimed.

Therefore $x^{\overline{W}_m}=a^2a^{\overline{W}_m},$ as asserted.\end{proof}

\begin{cor} Let $a\in R^{\overline{W}_m}$. Then $[(a^{\overline{W}_m})^{\overline{W}_m}]^{\overline{W}_m}=a^{\overline{W}_m}.$
\end{cor}
\begin{proof} By virtue of Theorem 2.3, we have
$$\begin{array}{rll}
[(a^{\overline{W}_m})^{\overline{W}_m}]^{\overline{W}_m}&=&(a^{\overline{W}_m})^2
(a^{\overline{W}_m})^{\overline{W}_m}\\
&=&(a^{\overline{W}_m})^2[a^2a^{\overline{W}_m}]\\
&=&(a^{\overline{W}_m})[(a^{\overline{W}_m})a^2](a^{\overline{W}_m})\\
&=&a^{\overline{W}_m}aa^{\overline{W}_m}\\
&=&a^{\overline{W}_m}.
\end{array}$$ \end{proof}

\begin{cor} Let $a\in R^{\overline{W}_m}$. Then $a^{\overline{W}_m}\in R^{\overline{W}}$ and
$(a^{\overline{W}_m})^{\overline{W}}=a^2a^{\overline{W}_m}.$
\end{cor}
\begin{proof} Let $x=a^{\overline{W}_m}$. Then $x^{\overline{W}_m}=a^2x$ by Theorem 2.3.
Hence, $x[a^2x]^2=a^2x, [a^2x]x^n=x^{n+1}$ for some $n\in {\Bbb N}$. On the other hand, $xa^{m+1}=a^m$ for some $n\in {\Bbb N}$.
Hence, we have
$$\begin{array}{rll}
x^*x^2[a^2x]&=&x^*x(xa)ax=x^*x(xa^2x)\\
&=&x^*x(ax)=x^*x.
\end{array}$$ Hence, $[x^*x^2(a^2x)]^*=x^*x^2(a^2x)$. Therefore $x^{\overline{W}}=a^2a^{\overline{W}_m},$ as asserted. \end{proof}

Next we study additive property and reverse order law of the $m$-weak group inverse. The following lemma is crucial.

\begin{lem} Let $a\in R$. Then $a\in R^{\overline{W}_m}$ if and only if $a^m\in R^{\overline{W}}.$ In this case, $a^{\overline{W}_m}=a^{m-1}(a^m)^{\overline{W}}.$
\end{lem}
\begin{proof} See~[23, Proposition 4.6].\end{proof}

\begin{lem} Let $a, b\in R^{\overline{W}}$. If $ab=0, ba=0$ and $a^*b=0$, then $a+b\in R^{\overline{W}}$. In this case, $$(a+b)^{\overline{W}}=a^{\overline{W}}+b^{\overline{W}}.$$
\end{lem}
\begin{proof} See~[21, Theorem 5.3].\end{proof}

\begin{thm} Let $a,b \in R^{\overline{W}_m}$. If $ab=0, ba=0$ and $a^*b=0$, then $a+b\in R^{\overline{W}_m}$. In this case, $$(a+b)^{\overline{W}_m}=a^{\overline{W}_m}+b^{\overline{W}_m}.$$\end{thm}
\begin{proof} Since $a,b\in R^{\overline{W}_m}$, by virtue of Lemma 2.7, we have
 $a^m,b^m\in R^{\overline{W}}$. By hypothesis, we easily check that
 $$a^mb^m=0, b^ma^m=0, (a^m)^*b^m=0.$$ According to Lemma 2.8, $$(a+b)^m=a^m+b^m\in R^{\overline{W}}.$$ By using Lemma 2.7,
 $a+b\in R^{\overline{W}_m}$. Moreover, we have $$\begin{array}{rll}
 (a+b)^{\overline{W}_m}&=&(a+b)^{m-1}[(a+b)^m]^{\overline{W}}\\
 &=&(a+b)^{m-1}[a^m+b^m]^{\overline{W}}\\
 &=&(a+b)^{m-1}[(a^m)^{\overline{W}}+(b^m)^{\overline{W}}]\\
 &=&a^{m-1}(a^m)^{\overline{W}}+b^{m-1}(b^m)^{\overline{W}}\\
&=&a^{\overline{W}_m}+b^{\overline{W}_m},
\end{array}$$ as asserted.\end{proof}

\begin{lem} Let $a,b\in R^{\overline{W}}$. If $ab=ba, a^*b=ba^*$, then $ab\in R^{\overline{W}}$. In this case,
$$(ab)^{\overline{W}}=a^{\overline{W}}b^{\overline{W}}=b^{\overline{W}}a^{\overline{W}}.$$\end{lem}
\begin{proof} See~[21, Theorem 5.2].\end{proof}

\begin{thm} Let $a,b\in R^{\overline{W}_m}$. If $ab=ba, a^*b=ba^*$, then $ab\in R^{\overline{W}_m}$. In this case,
$$(ab)^{\overline{W}_m}=a^{\overline{W}_m}b^{\overline{W}_m}=b^{\overline{W}_m}a^{\overline{W}_m}.$$\end{thm}
\begin{proof} In view of Lemma 2.7, $a^m, b^m\in R^{\overline{W}}$.
Since $ab=ba, a^*b=ba^*$, we see that $a^mb^m=b^ma^m, (a^m)^*b^m=b^m(a^m)^*$. In view of Lemma 2.10,
$(ab)^m=a^mb^m\in \in R^{\overline{W}}$. By virtue of Lemma 2.7 again, $ab\in R^{\overline{W}_m}$.
Since $ba^m=a^mb$ and $b^*a^m=a^mb^*$, it follows by ~[21, Lemma 5.1] that
$b(a^m)^{\overline{W}}=(a^m)^{\overline{W}}b$.
Therefore we have $$\begin{array}{rll}
(ab)^{\overline{W}_m}&=&(ab)^{m-1}[(ab)^m]^{\overline{W}}\\
&=&(ab)^{m-1}[a^mb^m]^{\overline{W}}\\
&=&(ab)^{m-1}(a^m)^{\overline{W}}(b^m)^{\overline{W}}\\
&=&[a^{m-1}(a^m)^{\overline{W}}][b^{m-1}(b^m)^{\overline{W}}]\\
&=&a^{\overline{W}_m}b^{\overline{W}_m}.
\end{array}$$ Likewise, $(ab)^{\overline{W}_m}=b^{\overline{W}_m}a^{\overline{W}_m}$. This completes the proof.\end{proof}

\section{Equivalent characterizations}

In this section we establish various equivalent characterizations of $m$-weak group inverse. We now derive

\begin{thm} Let $a\in R$. Then $a\in R^{\overline{W}_m}$ if and only if\end{thm}
\begin{enumerate}
\item [(1)] $a\in R^D$;
\vspace{-.5mm}
\item [(2)] There exists $x\in R$ such that $$ax^2=x, [(a^m)^*a^{m+1}x]^*=(a^m)^*a^{m+1}x, a^n=axa^n$$ for some $n\in {\Bbb N}$.
\end{enumerate}
In this case, $a^{\overline{W}_m}=aa^Dx$.
\begin{proof} $\Longrightarrow $ In view of Theorem 2.1, $a\in R^D$. By hypothesis, there exists $x\in R$ such that
$$x=ax^2, [(a^m)^*a^{m+1}x]^*=(a^m)^*a^{m+1}x, a^n=xa^{n+1}~\mbox{for some} ~n\in {\Bbb N}.$$
Then $axa^n=(ax^2)a^{n+1}=xa^{n+1}=a^n$, as desired.

$\Longleftarrow $ By hypothesis, there exists some $z\in R$ such that
such that $$az^2=z, [(a^m)^*a^{m+1}z]^*=(a^m)^*a^{m+1}z, a^n=aza^n$$ for some $n\in {\Bbb N}$. Let $x=aa^Dz$.
We claim that $a^{\overline{W}_m}=x$. Let $k=ind(a)$. Then one directly verifies that
$$ax=a^Da(az)=a^Da(a^kz^k)=(a^Da^{k+1})z^k=a^kz^k=az.$$ This implies that $$[(a^m)^*a^{m+1}x]^*=[(a^m)^*a^{m+1}z]^*=(a^m)^*a^{m+1}z=(a^m)^*a^{m+1}x.$$
Moreover, we have
$$\begin{array}{rll}
x-ax^2&=&aa^Dz-(ax)x=aa^Dz-(az)aa^Dz\\
&=&(aa^Dz-az^2)+az(1-aa^D)z=(a^2a^Dz^2-az^2)\\
&+&az(1-aa^D)z=(aa^D-1)az^2+az(1-aa^D)z\\
&=&(az-1)(1-aa^D)z.
\end{array}$$ Since $z=az^2$, by induction, $z=z^kz^{k+1}$. Then
$x-ax^2=(az-1)(1-aa^D)a^kz^{k+1}=0$, and so $ax^2=x$.

Furthermore, $$\begin{array}{rll}
xa^{n+k+1}&=&aa^Dza^{n+k+1}=a^D(az)a^{n+k+1}=a^D(aza^n)a^{k+1}\\
&=&a^Da^na^{k+1}=[a^Da^{k+1}]a^n=a^{n+k},
\end{array}$$ as required.\end{proof}

\begin{cor} Let $A\in {\Bbb C}^{n\times n}$. Then $A^{\overline{W}_m}=X$ if and only if $$AX^2=X, [(A^m)^*A^{m+1}X]^*=(A^m)^*A^{m+1}X,
A^k=AXA^k$$ for some $k\in {\Bbb N}$.\end{cor}
\begin{proof} This is immediate by Theorem 3.1.\end{proof}

We are ready to prove:

\begin{thm} Let $a\in R$. Then $a\in R^{\overline{W}_m}$ if and only if\end{thm}
\begin{enumerate}
\item [(1)] $a\in R^D$;
\vspace{-.5mm}
\item [(2)] There exists $x\in R$ such that $$xR=a^DR, [(a^m)^*a^{m+1}x]^*=(a^m)^*a^{m+1}x, a^n=axa^{n}~\mbox{for some} n\in {\Bbb N}.$$
In this case, $a^{\overline{W}_m}=x$.\end{enumerate}
\begin{proof} $\Longrightarrow $  In view of Theorem 2.1, $a$ has the $m$-weak group decomposition $a=a_1+a_2$, where $a_1^*a^{m-1}a_2=a_2a_1=0, a_1\in R^{\#}, a_2\in R^{nil}.$ Let $x=a_1^{\#}$. Moreover, we have $$x=ax^2, [(a^m)^*a^{m+1}x]^*=(a^m)^*a^{m+1}x, a^n=xa^{n+1}~\mbox{for some} ~n\in {\Bbb N}.$$
Then it will suffice to prove that $xR=a^DR$.

Since $a_2a_1=0, a_1a_1^{\pi}=0$ and $a_2^D=0$, it follows by ~[9, Theorem 2.1] that $a\in R^D$ and $$a^D=a_1^{\#}+\sum\limits_{n=1}^{k-1
}(a_1^{\#})^{n+1}a_2^n,$$ where $k=ind(a_2)$.
It is easy to verify that $$\begin{array}{rll}
aa^Dx&=&a^D(a_1+a_2)a_1^{\#}=a^Da_1a_1^{\#}\\
&=&[a_1^{\#}+\sum\limits_{n=1}^{k-1}(a_1^{\#})^{n+1}a_2^n
]a_1a_1^{\#}\\
&=&a_1^{\#}(a_1a_1^{\#})=x.
\end{array}$$ Moreover, we have
$$\begin{array}{rll}
xaa^D&=&a_1^{\#}(a_1+a_2)[a_1^{\#}+\sum\limits_{n=1}^{k-1}(a_1^{\#})^{n+1}a_2^n]\\
&=&a_1^{\#}
a_1[a_1^{\#}+\sum\limits_{n=1}^{k-1}(a_1^{\#})^{n+1}a_2^n]\\
&=&a_1^{\#}+\sum\limits_{n=1}^{k-1}(a_1^{\#})^{n+1}a_2^n\\
&=&a^D.
\end{array}$$
Therefore $xR=a^DR$, as desired.

$\Longleftarrow $ By hypothesis, we have
$$axa^D=a[xa^{n}](a^D)^{n+1}=(axa^n)(a^D)^{n+1}=a^n(a^D)^{n+1}=a^D.$$ Hence,
$(1-ax)a^D=0$. Since $xR\subseteq a^DR$, we deduce that $(1-ax)x=0$, and so $x=ax^2$. According to Theorem 2.1, $a\in R^{\overline{W}_m}$. In this case, $a^{\overline{W}_m}=x$, as asserted.
\end{proof}

\begin{cor} Let $A\in {\Bbb C}^{n\times n}$. Then $A^{\overline{W}_m}=X$ if and only if there exists $k\in {\Bbb N}$ such that $$\mathcal{R}(X)=\mathcal{R}(A^D), [(A^m)^*A^{m+1}X]^*=(A^m)^*A^{m+1}X, A^k=AXA^{k}.$$\end{cor}
\begin{proof} This is obvious by Theorem 3.3.\end{proof}

We now characterize the $m$-weak group inverse by using a polar-like property.

\begin{thm} Let $a\in R$. Then $a\in R^{\overline{W}_m}$ if and only if\end{thm}
\begin{enumerate}
\item [(1)] $a\in R^D$;
\vspace{-.5mm}
\item [(2)] There exists an idempotent $p\in R$ such that $$a+p\in \mathcal{A}~\mbox{is right invertible},(a^D)^*a^mp=0~\mbox{and} ~a^DR=(1-p)R.$$
\end{enumerate}
In this case, $a^{\overline{W}_m}=a^D(1-p).$
\begin{proof}  $\Longrightarrow $ Obviously, $a\in R^D$. Let $q=aa^{\overline{W}_m}$. In view of Corollary 2.2, $a^{\overline{W}_m}=a^{\overline{W}_m}aa^{\overline{W}_m}.$ Then $q^2=q$, i.e., $q\in R$ is an idempotent.
Moreover, we have $a^{\overline{W}_m}= .$
One easily verifies that
$$\begin{array}{rll}
(a^m)^*a^mq&=&(a^m)^*a^{m+1}a^{\overline{W}_m}\\
&=&\big((a^m)^*a^{m+1}a^{\overline{W}_m}\big))^*\\
&=&\big((a^m)^*a^mq\big)^*=q^*(a^m)^*a^m.
\end{array}$$
Since $a^n=a^{\overline{W}_m}a^{n+1}$ for some $n\in {\Bbb N}$, we see that
$$qaa^D=aa^{\overline{W}_m}aa^D=a[a^{\overline{W}_m}a^{n+1}](a^D)^{n+1}=a^{n+1}(a^D)^{n+1}=aa^D.$$
Since $a^{\overline{W}_m}=a(a^{\overline{W}_m})^2$, analogously, we prove that
$$aa^Dq=aa^Daa^{\overline{W}_m}=q.$$ Thus we have
$$\begin{array}{rll}
(a^D)^*a^mq&=&[a^m(a^D)^{m+1}]^*a^mq=[(a^D)^{m+1}]^*[(a^m)^*a^mq]\\
&=&[(a^D)^{m+1}]^*[q^*(a^m)^*a^m]=[(a^D)^{m+1}]^*[(a^m)^*a^mq]^*\\
&=&[(a^D)^{m+1}]^*[(a^m)^*a^{m+1}a^Dq]^*\\
&=&[(a^D)^{m+1}]^*[aa^D]^*[a^{m+1}a^Dq]^*a^m\\
&=&[(a^D)^{m+1}]^*[a^{m+1}a^Dqaa^D]^*a^m\\
&=&[(a^D)^{m+1}]^*[a^{m+1}a^Daa^D]^*a^m=(a^D)^*a^m.
\end{array}$$ Thus, $(a^D)^*a^m(1-q)=(a^D)^*a^m-(a^D)^*a^mq=0$.

Let $a=x+y$ be the $m$-weak decomposition of $a$. Then $a^{\overline{W}_m}=x^{\#}$. Hence,
$a[1-aa^{\overline{W}_m}]=(x+y)[1-(x+y)x^{\#}]=(x+y)[1-xx^{\#}]=y$ is nil.
Then we check that
$$\begin{array}{rl}
&[a+1-aa^{\overline{W}_m}][a^{\overline{W}_m}+1-aa^{\overline{W}_m}]\\
=&aa^{\overline{W}_m}+a[1-aa^{\overline{W}_m}]+[1-aa^{\overline{W}_m}]a^{\overline{W}_m}+[1-aa^{\overline{W}_m}]\\
=&1+a[1-aa^{\overline{W}_m}]\\
\in&\mathcal{A}^{-1}.\\
\end{array}$$
Then $a+1-q=a+1-aa^{\overline{W}_m}\in \mathcal{A}$ is right invertible. In view of Theorem 3.3, $a^DR=a^{\overline{W}_m}R=aa^{\overline{W}_m}R=qR$. Choose $p=1-q$. Then $$a+p\in \mathcal{A}~\mbox{is right invertible}, (a^D)^*a^mp=0~\mbox{and} ~a^DR=(1-p)R,$$ as required.

$\Longleftarrow $ By hypothesis, there exists an idempotent $p\in R$ such that
$$a+p\in \mathcal{A}~\mbox{right invertible},(a^D)^*a^mp=0~\mbox{and} ~a^DR=(1-p)R.$$
Set $q=1-p$. Then $$(a^D)^*a^mq=(a^D)^*a^m~\mbox{and} ~a^DR=qR.$$

Since $a^mq\in a^DR$, we can find $x\in R$ such that $a^mq=a^Dx$.
Hence, $(a^D)^*a^Dx=(a^D)^*a^m$.
Then $$\begin{array}{rll}
(aa^D)^*aa^D&=&a^*[(a^D)^*a^m](a^D)^m=a^*[(a^D)^*a^Dx](a^D)^m\\
&=&(aa^D)^*aa^D[a^Dx(a^D)^m].
\end{array}$$ Since the involution $*$ is proper, we derive
that $aa^D=a^Dx(a^D)^m$. Choose $y=(a^D)^{m+2}x$. Then we verify that
$$\begin{array}{rll}
ay^2&=&a(a^D)^{m+2}x(a^D)^{m+2}x=(a^D)^{m+1}x(a^D)^{m+2}x\\
&=&(a^D)^m[a^Dx(a^D)^m](a^D)^2x\\
&=&(a^D)^m(aa^D)(a^D)^2x=(a^D)^{m+2}x=y,\\
(a^D)^*a^{m+1}y&=&(a^D)^*a^{m+1}(a^D)^{m+2}x=(a^D)^*a^Dx=(a^D)^*a^m.\\
\end{array}$$
Set $k=ind(a)$. Then we get $$(a^D)^*[a^Dxa^{k+1}]=[(a^D)^*a^Dx]a^{k+1} =[(a^D)^*a^m]a^{k+1}=(a^D)^*(a^ma^{k+1}).$$
Thus, $$(a^D)^*[a^Dxa^{k+1}]=(a^D)^*[a^Da^{m+1}a^{k+1}].$$
As the involution is proper, we get $a^Dxa^{k+1}=(a^Da^{k+1})a^{m+1}=a^ma^{k+1}.$ Accordingly, we deduce that
$$\begin{array}{rll}
ya^{k+1}&=&[(a^D)^{m+2}x]a^{k+1}=(a^D)^{m+1}[a^Dxa^{k+1}]\\
&=&(a^D)^{m+1}[a^ma^{k+1}]=a^Da^{k+1}=a^k.
\end{array}$$ In light of Theorem 2.1, $a\in R^{\overline{W}_m}$, as asserted.\end{proof}

\begin{cor} Let $a\in R$. Then $a\in R^{\overline{W}}$ if and only if\end{cor}
\begin{enumerate}
\item [(1)] $a\in R^D$;
\vspace{-.5mm}
\item [(2)] There exists an idempotent $p\in R$ such that $$a+p\in \mathcal{A}~\mbox{is right invertible},(a^D)^*ap=0~\mbox{and} ~a^DR=(1-p)R.$$
\end{enumerate}
In this case, $a^{\overline{W}}=a^D(1-p).$
\begin{proof} This is obvious by choosing $m=1$ in Theorem 3.5.\end{proof}

Let $A\in {\Bbb C}^{n\times n}$ have group inverse. Then there exists $E=E^2\in {\Bbb C}^{n\times n}$ such that $A+E~\mbox{is invertible},AE=EA=0$ (see~[1, Proposition 8.24]. Contract to this fact, we now derive

\begin{cor} Let $A\in {\Bbb C}^{n\times n}$. Then there exists $E=E^2\in {\Bbb C}^{n\times n}$ such that $$A+E~\mbox{is invertible},(A^D)^*AE=0~\mbox{and} ~\mathcal{R}(A^D)=\mathcal{R}(I_n-E).$$\end{cor}

\section{relations with core-EP inverses}

In this section we investigate relations between $m$-weak group inverse and core-EP inverse. We now derive

\begin{thm} Let $a\in R^{\overline{D}}$. Then $a\in R^{\overline{W}_m}$ and $$a^{\overline{W}_m}=(a^{\overline{D}})^{m+1}a^m.$$\end{thm}
\begin{proof} By hypothesis, we have $$a^{\overline{D}}=a(a^{\overline{D}})^2, (aa^{\overline{D}})^*=aa^{\overline{D}}~\mbox{and} ~a^n=a^{\overline{D}}a^{n+1}$$ for some $n\in {\Bbb N}$. Set
$x=(a^{\overline{D}})^{m+1}a^m$. Then
$$\begin{array}{rll}
ax^2&=&a[(a^{\overline{D}})^{m+1}a^m]^2=[(a^{\overline{D}})^ma^m][(a^{\overline{D}})^{m+1}a^m]\\
&=&(a^{\overline{D}})^m[a^m(a^{\overline{D}})^{m+1}]a^m=(a^{\overline{D}})^{m+1}a^m=x,\\
(a^m)^*a^{m+1}x&=&(a^m)^*a^{m+1}[(a^{\overline{D}})^{m+1}a^m]=(a^m)^*[aa^{\overline{D}}]a^m,\\
\big((a^m)^*a^{m+1}x\big)^*&=&(a^m)^*[aa^{\overline{D}}]^*a^m=(a^m)^*[aa^{\overline{D}}]a^m=(a^m)^*a^{m+1}x,\\
a^n&=&a^{\overline{D}}a^{n+1}=a^{\overline{D}}[a^{\overline{D}}a^{n+1}]a=(a^{\overline{D}})^2a^{n+2}\\
&=&\cdots =(a^{\overline{D}})^{m+1}a^{m+n+1}=xa^{n+1}.\end{array}$$
Therefore $a^{\overline{W}_m}=(a^{\overline{D}})^{m+1}a^m.$\end{proof}

\begin{cor} Let $a\in R^{\overline{D}}$ with $ind(a)=k$. Then the following are equivalent:\end{cor}
\begin{enumerate}
\item [(1)] $a^{\overline{W}_m}=x$.
\vspace{-.5mm}
\item [(2)] $ax^2=x$ and $ax=(a^{\overline{D}})^ma^m.$
\vspace{-.5mm}
\item [(3)] $im(x)=im(a^k)$ and $ax=(a^{\overline{D}})^ma^m.$
\end{enumerate}
\begin{proof} $(1)\Rightarrow (2)$ By virtue of Theorem 4.1, $a\in R^{\overline{W}_m}$ and $x=a^{\overline{W}_m}=(a^{\overline{D}})^{m+1}a^m$.
Accordingly, $ax^2=x$ and $ax=a(a^{\overline{D}})^{m+1}a^m=(a^{\overline{D}})^ma^m$, as desired.

$(2)\Rightarrow (3)$ Since $ax^2=x$, we have $x=a^kx^k$, and so $im(x)\subseteq im(a^k)$. On the other hand, $a^k=xa^{k+1}\subseteq im(x)$, and so
$im(a^k)\subseteq im(x)$, as desired.

$(3)\Rightarrow (1)$ By hypotheses, $im(x)=im(a^k)$ and $ax=(a^{\overline{D}})^ma^m.$
Write $x=a^kr$ for some $r\in R$. Then $$\begin{array}{rll}
ax^2&=&(ax)x=(a^{\overline{D}})^ma^{m+k}r\\
&=&(a^{\overline{D}})^{m-1}[a^{\overline{D}}a^{k+1}]a^{m-1}r\\
&=&(a^{\overline{D}})^ma^ka^{m-1}r\\
&=&(a^{\overline{D}})^{m-1}a^{m-1+k}r=\cdots =a^kr=x.
\end{array}$$ Thus $x=(ax)x=[(a^{\overline{D}})^ma^m]x=(a^{\overline{D}})^ma^{m-1}(ax)
=(a^{\overline{D}})^ma^{m-1}[(a^{\overline{D}})^ma^m]=(a^{\overline{D}})^ma[(a^{\overline{D}})^2
=(a^{\overline{D}})^{m+1}a^m$. In light of Theorem 4.1, $x=a^{\overline{W}_m}$, as asserted.\end{proof}

\begin{cor} Let $a\in R^{\overline{D}}$ with $ind(a)=k$. Then the following are equivalent:\end{cor}
\begin{enumerate}
\item [(1)] $a^{\overline{W}_m}=x$.
\vspace{-.5mm}
\item [(2)] $xax=x, im(x)=im(a^k)$ and $(a^k)^*a^{m+1}x=(a^k)^*a^m.$
\vspace{-.5mm}
\item [(3)] $xax=x, im(x)=im(a^k)$ and $(a^D)^*a^{m+1}x=(a^D)^*a^m.$
\end{enumerate}
\begin{proof} $(1)\Rightarrow (2)$ By hypothesis, there exists $x\in R$ such that $$ax^2=x, (a^k)^*a^{m+1}x=(a^k)^*a^m~\mbox{and} ~a^k=xa^{k+1}.$$
Then $xax=xa^{k+1}x^{k+1}=a^kx^{k+1}=ax^2=x$. As $x=ax=a^kx^{k+1}$, we see that $im(x)=im(a^k)$, as required.

$(2)\Rightarrow (3)$ By hypothesis, we have $x\in R$ such that
$xax=x, im(x)=im(a^k)$ and $(a^k)^*a^{m+1}x=(a^k)^*a^m.$
Then $$((a^D)^{k+1})^*(a^k)^*a^{m+1}x=((a^D)^{k+1})^*(a^k)^*a^m.$$
Therefore $(a^D)^*a^{m+1}x=(a^D)^*a^m,$ as desired.

$(3)\Rightarrow (1)$ Since $im(x)=im(a^k)$, there exists $z\in R$ such that $x=a^kz$.
Since $(a^D)^*a^{m+1}x=(a^D)^*a^m,$ we have $[(a^m)^D]^*a^{m+1}x=[(a^m)^D]^*a^m.$ Then
$$\begin{array}{rll}
[(a^m)^D]^*(a^m)^D[a^{2m+k+1}z]&=&[(a^m)^D]^*[(a^m)^Da^ma^Da^{k+1}]a^{m+1}z\\
&=&[(a^m)^D]^*a^{m+1}a^kz=[(a^m)^D]^*a^{m+1}x\\
&=&[(a^m)^D]^*a^m.
\end{array}$$ In view of ~[3, Theorem 5.4], we have
$$(a^m)^{\overline{W}}=[(a^m)^D]^3[a^{2m+k+1}z]=(a^m)^Da^{k+1}z=(a^m)^Dax.$$
In light of Lemma 2.7, $a^{\overline{W}_m}=a^{m-1}(a^m)^{\overline{W}}=a^{m-1}(a^m)^Dax=aa^Dx=a^Da^{k+1}z=a^kz=x,$ as required.\end{proof}

\begin{cor} Let $a\in R^{\overline{D}}$. Then the following are equivalent:\end{cor}
\begin{enumerate}
\item [(1)] $a^{\overline{W}_m}=a^{\overline{D}}$;
\vspace{-.5mm}
\item [(2)] $aa^{\overline{D}}=(a^{\overline{D}})^ma^m$.
\vspace{-.5mm}
\item [(3)] $aa^{\overline{W}_m}=aa^{\overline{D}}$.\end{enumerate}
\begin{proof} $(1)\Rightarrow (2)$ In view of Theorem 4.1, $$a^{\overline{W}_m}=(a^{\overline{D}})^{m+1}a^m.$$ Hence,
$aa^{\overline{D}}=aa^{\overline{W}_m}=a(a^{\overline{D}})^{m+1}a^m=(a^{\overline{D}})^ma^m$, as desired.

$(2)\Rightarrow (1)$ Since $aa^{\overline{D}}=(a^{\overline{D}})^ma^m$, it follows by Theorem 4.1 that $$\begin{array}{rll}
a^{\overline{W}_m}&=&(a^{\overline{D}})^{m+1}a^m\\
&=&a^{\overline{D}}[(a^{\overline{D}})^ma^m]\\
&=&a^{\overline{D}}aa^{\overline{D}}\\
&=&a^{\overline{D}},
\end{array}$$ as desired.

$(1)\Rightarrow (3)$ This is trivial.

$(3)\Rightarrow (1)$ By virtue of Theorem 4.1, we have $a^{\overline{W}_m}=(a^{\overline{D}})^{m+1}a^m.$ Accordingly, we check that
$$\begin{array}{rll}
a^{\overline{W}_m}&=&(a^{\overline{D}})^{m+1}a^m\\
&=&aa^D(a^{\overline{D}})^{m+1}a^m\\
&=&a^D(aa^{\overline{W}_m})\\
&=&a^D(aa^{\overline{D}})\\
&=&a^{\overline{D}},
\end{array}$$ as asserted.\end{proof}

The GG inverse of $A$ is defined to the unique solution to the following system: $$AX=A^{\overline{D}}A^2, \mathcal{R}(X)\subseteq \mathcal{R}(A^k).$$ For properties of GG inverse, we refer the reader to ~[4]. We come now to provide a new characterization for the GG inverse of a complex matrix.

\begin{cor} Let $A,X\in {\Bbb C}^{n\times n}$. Then the following are equivalent:\end{cor}
\begin{enumerate}
\item [(1)] $A$ has GG-inverse $X$.
\item [(2)] $X=(A^{\overline{D}})^3A^2$.
\item [(3)] $X=(A^{\overline{W}})^2A$.
\item [(4)] There exist $B,C\in {\Bbb C}^{n\times n}$ such that $$A=Y+Z, Y^*AZ=ZY=0, Y~\mbox{has group inverse}, Z~\mbox{is nilpotent}.$$
\item [(5)] There exists $k\in {\Bbb N}$ such that
$$X=AX^2, (A^D)^*A^3X=(A^D)^*A^2, A^k=XA^{k+1}.$$
\end{enumerate}
\begin{proof} $(1)\Rightarrow (2)$ This is proved in Theorem 4.1.

$(2)\Rightarrow (3)$ In view of Theorem 4.1, $A^{\overline{W}}=(A^{\overline{D}})^2A$. Then
$X=(A^{\overline{D}})^3A^2=(A^{\overline{D}})^2[AA^{\overline{D}})^2]A^2=(A^{\overline{W}})^2A$, as required.

$(3)\Rightarrow (1)$ Let $X=(A^{\overline{W}})^2A$. Then
$$X=[(A^{\overline{D}})^2A]^2A=(A^{\overline{D}})^3A^2.$$ In light of Theorem 4.1, we prove that
$X$ is the GG inverse of $A$, as desired.

$(1)\Leftrightarrow (4)\Leftrightarrow (5)$ These are immediate by Theorem 2.1.\end{proof}

We come now to present the representation of the $m$-weak inverse of a core-EP invertible element.

\begin{thm} Let $a\in R^{\overline{D}}$ and let $p=aa^{\overline{D}}$. Then
$$a=\left(
\begin{array}{cc}
t&s\\
0&n
\end{array}
\right)_p, a^{\overline{W}_m}=
\left(
\begin{array}{cc}
t^{-1}&t^{-(m+1)}c_m\\
0&0
\end{array}
\right)_p,$$ where $c_1=s, c_{i+1}=tc_i+sn^i ~(i\in {\Bbb N})$.\end{thm}
\begin{proof} Since $p=p^2=p^*$, we have a Pierce decomposition of $a$:
$a=\left(
\begin{array}{cc}
t&s\\
a_{21}&n
\end{array}
\right)_p$, where $a_{21}=(1-p)ap$. Set $k=ind(a)$. Then we have $$\begin{array}{rll}
a_{21}&=&(1-aa^{\overline{D}})a^2a^{\overline{D}}\\
&=&(1-aa^{\overline{D}})a^{k+1}(a^{\overline{D}})^k\\
&=&[a^{k+1}-a(a^{\overline{D}}a^{k+1})](a^{\overline{D}})^k\\
&=&0.
\end{array}$$
We easily see that $$a^{\overline{D}}=\left(
\begin{array}{cc}
t^{-1}&0\\
0&0
\end{array}
\right)_p.$$ According to Theorem 4.1, $$\begin{array}{rll}
a^{\overline{W}_m}&=&(a^{\overline{D}})^{m+1}a^m\\
&=&\left(
\begin{array}{cc}
t^{-1}&0\\
0&0
\end{array}
\right)_p^{m+1}\left(
\begin{array}{cc}
t&s\\
0&n
\end{array}
\right)_p^m\\
&=&\left(
\begin{array}{cc}
t^{-(m+1)}&0\\
0&0
\end{array}
\right)_p\left(
\begin{array}{cc}
t^m&c_m\\
0&n^m
\end{array}
\right)_p\\
&=&\left(
\begin{array}{cc}
t^{-1}&t^{-(m+1)}c_m\\
0&0
\end{array}
\right)_p,
\end{array}$$
as asserted.\end{proof}

\begin{cor} Let $a\in R^{\overline{D}}$. Then the following are equivalent:\end{cor}
\begin{enumerate}
\item [(1)] $a^{\overline{W}_m}=a^{\overline{D}}$.
\vspace{-.5mm}
\item [(2)] $c_m=0$, where $c_1=s, c_{i+1}=tc_i+sn^i ~(i\in {\Bbb N})$.
\end{enumerate}
\begin{proof} By virtue of Theorem 4.6, $a^{\overline{W}_m}=a^{\overline{D}}$ if and only if $c_m=0$.
This completes the proof.\end{proof}

\begin{cor} Let $a\in R^{\overline{D}}$ and $m<n$. Then the following are equivalent:\end{cor}
\begin{enumerate}
\item [(1)] $[a^{\overline{W}_m}]^n=(a^n)^{\overline{W}}$.
\vspace{-.5mm}
\item [(2)] $tc_m-c_mn=t^{11}s$, where $c_1=s, c_{i+1}=tc_i+sn^i ~(i\in {\Bbb N})$.
\end{enumerate}
\begin{proof} Let $p=aa^{\overline{D}}\in R^{\overline{D}}$. In view of Theorem 4.6, we have
$$a=\left(
\begin{array}{cc}
t&s\\
0&n
\end{array}
\right)_p, a^{\overline{W}_m}=
\left(
\begin{array}{cc}
t^{-1}&t^{-(m+1)}c_m\\
0&0
\end{array}
\right)_p,$$ where $c_1=s, c_{i+1}=tc_i+sn^i ~(i\in {\Bbb N})$. Then
$$\begin{array}{rll}
aa^{\overline{W}_m}&=&\left(
\begin{array}{cc}
t&s\\
0&n
\end{array}
\right)\left(
\begin{array}{cc}
t^{-1}&t^{-(m+1)}c_m\\
0&0
\end{array}
\right)\\
&=&\left(
\begin{array}{cc}
p&t^{-m}c_m\\
0&0
\end{array}
\right),\\
a^{\overline{W}_m}a&=&\left(
\begin{array}{cc}
t^{-1}&t^{-(m+1)}c_m\\
0&0
\end{array}
\right)\left(
\begin{array}{cc}
t&s\\
0&n
\end{array}
\right)\\
&=&\left(
\begin{array}{cc}
p&t^{-1}s+t^{-(m+1)}c_mn\\
0&0
\end{array}
\right)_p.
\end{array}$$ By virtue of ~[23, Theorem 4.13], $[a^{\overline{W}_m}]^n=(a^n)^{\overline{W}}$ if and only if
$aa^{\overline{W}_m}=a^{\overline{W}_m}a$ if and only if $t^{-m}c_m=t^{-1}s+t^{-(m+1)}c_mn$, i.e., $tc_m-c_mn=t^{11}s$.
This completes the proof.\end{proof}

\end{document}